\documentclass[11pt,leqno]{amsart}

\usepackage{amsmath}
\usepackage{amssymb}
\usepackage{a4wide}
\usepackage{graphics}
\usepackage{epsfig}
\usepackage{enumerate}
\newtheorem{theorem}{Theorem}
\newtheorem{lemma}[theorem]{Lemma}
\newtheorem{proposition}[theorem]{Proposition}

\newtheorem{remark}[theorem]{Remark} 
\newtheorem{corollary}[theorem]{Corollary}

\newcounter{hypo}

\def\C{{\mathbb C}}

\def\N{{\mathbb N}} 
\def\R{{\mathbb R}} 

\def\Z{{\mathbb Z}}

\def\CH{\mathcal {H}}

\def\CL{\mathcal {L}}

\def\CO{\mathcal {O}}

 \def\im{\mathop{\rm Im}\nolimits}

\def\dive{\mathop{\rm div}\nolimits}

\def\<{\langle}
\def\>{\rangle}

\def\ds{\displaystyle}

\makeatletter
 \@addtoreset{equation}{section}
 \makeatother

\author[J.-F. Bony]{Jean-Fran\c{c}ois Bony}
\address{Jean-Fran\c{c}ois Bony, Institut de Math\'ematiques de Bordeaux, UMR 5251 du CNRS, Universit\'e de Bordeaux I, 351 cours de la Lib\'eration, 33405 Talence cedex, France}
\email{bony@math.u-bordeaux1.fr}

\author[D. H\"{a}fner]{Dietrich H\"{a}fner}
\address{Dietrich H\"{a}fner, Universit\'e de Grenoble 1, Institut Fourier, UMR 5582 du CNRS, BP 74, 38402 St Martin d'H\`eres, France}
\email{Dietrich.Hafner@ujf-grenoble.fr}

\title[Local energy decay in odd dimensions]{Improved local energy decay for the wave equation on asymptotically Euclidean odd dimensional manifolds in the short range case}



\begin{document}

\begin{abstract}
We show improved local energy decay for the wave equation on asymptotically Euclidean manifolds in odd dimensions in the short range case. The precise decay rate depends on the decay of the metric towards the Euclidean metric. We also give estimates of powers of the resolvent of the wave propagator between weighted spaces.
\end{abstract}

\maketitle

\section{Introduction}

The aim of this paper is to investigate the decay of the local energy for the wave equation associated to short range metric perturbations of the Euclidean Laplacian on $\R^{d}$, $d \geq 3$ odd. More precisely, for any $\rho > 0$, we show that the local energy decays like $\< t \>^{- \rho}$ if the metric converges like $\< x \>^{- \rho - 2 - \varepsilon}$ toward the Euclidean metric. This result rests on the $C^{\rho + 1}$ smoothness of the weighted resolvent of the wave generator.

The case of the wave equation in dimension $d \geq 3$ odd is very specific. Indeed, in flat space, the strong Huygens principle guaranties that the local energy decays as fast as we want. For compactly supported perturbations, this no longer holds in general but one can use the theory of resonances (see \cite{Sj07_01} for a general presentation of this field) to prove dispersive estimates. In non-trapping situations, this theory gives a resonance expansion of the cut-off propagator which implies an exponential decay of the local energy with an optimal decay rate as in \cite{LaPh89_01} for example. Such properties are related to the meromorphic extension to the whole complex plane (and, in particular, in a neighborhood of $0$) of the cut-off resolvent of the wave generator (see \cite{SjZw91_01}, \cite{Va89_01}). The resonance theory can also be used in trapping situations, but there is necessarily a ``loss of derivatives'' in the local energy estimate, see \cite{Ra69_01}. Among the large literature on this subject, we only refer to \cite{Bu98_01}, \cite{TaZw00_01}.

One can also obtain exponential decay of the local energy using the theory of resonances for exponentially decaying perturbations. In this case, the weighted resolvent has a meromorphic extension only in a half-plane containing the real axis and the exponential decay rate of the local energy is controlled by the exponential decay of the perturbation at infinity. Such ideas were developed in \cite{DoMcTh66_01}, \cite{Ki11_01}, \cite{SaZw95_01}. It is therefore natural to ask what are the decay rate and the regularity properties of the resolvent for polynomially decaying perturbations. In such situations, it is unlikely that the resolvent is analytic near the real axis. However, we might hope that the weighted resolvent has some $C^{k}$ regularity properties up to the real line, depending on the decay rate of the perturbation. In the same way, the exponential decay should be replaced by a polynomial one. In this paper, we show that this is indeed the case.

Note that the definition of the resonances by complex dilation or distortion (see \cite{AgCo71_01}, \cite{Hu86_01}) does not seem to be appropriate to show local energy decay at low frequencies. Indeed, such methods do not give good estimates of the resolvent near the thresholds. Concerning the resonances, we also mention that the dynamically definition of \cite{GeSi92_01} which describes the long time evolution of well-prepared initial data.

To prove the local energy decay, one can also apply other techniques like the vector field methods (in the huge literature of this field, see e.g. \cite{Kl01_01} and the books \cite{Al10_01}, \cite{Ho97_01}), the Mourre theory (see \cite{BoHa10_03}, \cite{Bo10_01}), \ldots However, in general, these methods do not distinguish between the parity of the dimension and give the polynomial decay of the local energy that one expects in even dimensions mutatis mutandis. Eventually, the theory of perturbations can be used to get resolvent estimates at low energy and then decay of the local energy for ``small perturbations'' (short range interactions, lower order terms, \ldots). This approach, close to the one developed in this paper, has been followed in numerous papers concerning the local energy decay for the Schr\"{o}dinger equation perturbed by a potential (see \cite{JeKa79_01}, \cite{Ra78_01} for example).

In this paper, we consider the following operator on $\R^d$, with $d \geq 3$ odd,
\begin{equation} \label{a5}
P= - b \dive ( A \nabla b ) = - \sum_{i,j=1}^{d} b(x) \frac{\partial \ }{\partial x_{i}} A_{i,j} (x) \frac{\partial \ }{\partial x_{j}} b (x) ,
\end{equation}
where $b(x)\in C^{\infty}(\R^d)$ and $A (x)\in C^{\infty}(\R^d;\R^{d\times d})$ is a real symmetric $d\times d$ matrix. The $C^{\infty}$ hypothesis is made mostly for convenience, much weaker regularity could actually be considered. We make an ellipticity assumption:
\begin{equation} \tag{H1} \label{a3}
\exists \delta > 0 , \ \forall x \in \R^d \qquad A (x)\geq \delta I_d \quad \text{and} \quad b(x) \geq \delta , 
\end{equation}
$I_d$ being the identity matrix on $\R^{d}$. We also assume that $P$ is a long range perturbation of the Euclidean Laplacian:
\begin{equation} \tag{H2} \label{a4}
\exists \rho > 0 , \ \forall \alpha \in \N^d \qquad \vert \partial^{\alpha}_x ( A (x) - I_d ) \vert + \vert \partial^{\alpha}_x ( b(x) - 1 ) \vert\lesssim \<x\>^{-\rho-\vert\alpha\vert}. \\
\end{equation}

In particular, if $b=1$, we are concerned with an elliptic operator in divergence form $P = - \dive ( A \nabla )$. On the other hand, if $A = (g^2 g^{i,j}(x))_{i,j},\, b=(\det g^{i,j})^{1/4},\, g=\frac{1}{b}$, then the above operator is unitarily equivalent to the Laplace--Beltrami $- \Delta_{\mathfrak{g}}$ on $(\R^d, \mathfrak{g})$ with metric
\[\mathfrak{g} = \sum_{i,j=1}^{d} g_{i,j} (x) \, d x^i \, d x^j ,\]
where $(g_{i,j})_{i,j}$ is inverse to $(g^{i,j})_{i,j}$ and the unitary transform is just multiplication by $g$. We are mainly interested in the low frequency behaviour, but our result is global in energy if we suppose in addition
\begin{equation} \tag{H3} \label{a51}
P \text{ is non-trapping.}
\end{equation}
In the following, $\Vert \cdot \Vert$ will design the norm on $L^{2} ( \R^{d} )$ or ${\mathcal L}(L^2)$.
Let $H^s$ be the usual Sobolev space on $\R^d$. Then it is well known that $(P,D(P)=H^2)$ is selfadjoint on $L^2$. Let us first rewrite the wave equation associated to $P$ as a first order equation. The wave equation
\begin{equation}  \label{1.1}
\left\{\begin{aligned}
&(\partial_t^2 + P )u = 0 ,   \\
&u (0) = u_0 , \\
&\partial_t u (0) = u_1 ,
\end{aligned} \right.
\end{equation}
is equivalent to the first order equation
\begin{equation} \label{1.2}
\left\{ \begin{aligned}
&i \partial_t \psi = G \psi ,   \\
&\psi (0) = ( u_0 , u_1 ) ,
\end{aligned} \right.
\end{equation}
with $\psi=(u,\partial_t u)$ and
\begin{equation} \label{c9}
G=i\left(\begin{array}{cc} 0 & 1 \\ -P & 0\end{array} \right) .
\end{equation}
We also put $P_0=-\Delta$ and
\begin{equation} \label{c27}
G_0=i\left(\begin{array}{cc} 0 & 1 \\ -P_0 & 0\end{array} \right).
\end{equation}
Let $\dot{H}^1_P$ (resp. $\dot{H}^2_P$) be the completion of $C_0^{\infty}(\R^d)$ in the norm $\Vert u \Vert^{2}_{\dot{H}^1_P} = \< P u , u \>$ (resp. $\Vert u \Vert^{2}_{\dot{H}^2_P} = \< P u , u \> + \Vert P u \Vert^2$). Then it is well known that $(G, ( \dot{H}_P^2 \oplus \dot{H}_P^1 ))$ is selfadjoint on ${\mathcal E}=\dot{H}^1_P\oplus L^2$. We put ${\mathcal H}^s:=H^{s+1}\oplus H^s$. Our main result is the following:

\begin{theorem}\sl \label{b17}
Assume $d \geq 3$ odd, $\mu \geq 0$ and $\rho > \mu + 2$ ($\rho > \mu + 1$ in dimension $d=3$).

$i)$ For all $\chi \in C^{\infty}_{0} ( \R )$ and $\varepsilon > 0$, we have
\begin{equation*}
\Big\Vert \<x\>^{-\mu - 1 - \varepsilon}e^{- i t G} \chi(G) \<x\>^{-\mu - 1 - \varepsilon}\Big\Vert_{{\mathcal L}({\mathcal H}^s)}\lesssim\<t\>^{-\mu} .
\end{equation*}

$ii)$ If we suppose in addition \eqref{a51}, then the above estimate holds globally in energy:
\begin{equation*}
\Big\Vert \<x\>^{-\mu - 1 - \varepsilon} e^{- i t G} \<x\>^{-\mu - 1 - \varepsilon}\Big\Vert_{{\mathcal L}({\mathcal H}^s)}\lesssim\<t\>^{-\mu} .
\end{equation*}
For $d=3$, we can replace $\<x\>^{-\mu-1-\varepsilon}$ by $\<x\>^{-\mu-1/2-\varepsilon}$ in the above estimates. 
\end{theorem}

\begin{remark}\sl
Combining the previous theorem with \cite{BoHa10_03} and an interpolation argument, we can replace $\< x \>^{- \mu - 1 - \varepsilon}$ by $\< x \>^{- \mu - \varepsilon}$ in Theorem \ref{b17} if $\rho = + \infty$ and $\mu > 1$.
\end{remark}

Note that one can express the wave propagator at low frequencies in terms of $P$ using the classical formula
\begin{equation} \label{b5}
e^{- i t G} = \left( \begin{array}{cc} \cos t \sqrt{P} & \ds \frac{\sin t \sqrt{P}}{ \sqrt{P}}  \\
- \sqrt{P} \sin t \sqrt{P} & \cos t \sqrt{P} \end{array} \right) ,
\end{equation}
and that $\chi ( G ) = \chi ( \sqrt{P} ) \oplus \chi ( \sqrt{P} )$ for $\chi$ even. The proof of Theorem \ref{b17} rests on the following smoothness property of the weighted resolvent at low frequencies (see also the H\"{o}lder regularity stated in Proposition \ref{c19}).

\begin{theorem}\sl \label{b16}
Assume $d \geq 3$ odd, $k \in \N^{*}$ and $\rho > k + 1$ ($\rho > k$ for $d = 3$). Let $\kappa=k$ ($\kappa = k - 1 / 2$ for $d=3$, $k \geq 2$). Then, for all $s \in \R$ and $C , \varepsilon > 0$, we have
\begin{equation*}
\sup_{z \in \C \setminus \R , \, \vert z \vert \leq C} \big\Vert \< x \>^{- \kappa - \varepsilon} ( G - z )^{-k} \< x \>^{- \kappa - \varepsilon} \big\Vert_{\CL ( \CH^s, \CH^{s + k} )} \lesssim 1 .
\end{equation*}
\end{theorem}

The polynomial decay of the local energy (resp. the $C^{k}$ smoothness of the weighted resolvent) for polynomially decaying perturbations is analogous to the exponential decay of the local energy (resp. the analytic extension of the resolvent) given by the resonance theory for compactly supported or exponentially decaying perturbations.

\section{The free resolvent}  \label{s2}

The goal of this section is to show the following estimate on the free resolvent.

\begin{proposition}\sl \label{b1}
Let $d \geq 3$ be odd. For all $k\in \N$, $s \in \R$ and $C , \varepsilon >0$, we have
\begin{equation*}
\sup_{z \in \C \setminus \R , \ \vert z \vert \leq C} \big\Vert \< x \>^{- k - \varepsilon} ( G_{0} - z )^{-k} \< x \>^{- k - \varepsilon} \big\Vert_{\CL ( \CH^s , \CH^{s + k} )} \lesssim 1 .
\end{equation*}
\end{proposition}

To prove this result, we will write the free resolvent as an integral in time over the evolution and then use the strong Huygens principle. Note that we estimate the powers of the resolvent in a scale of Sobolev spaces rather than in a scale of energy spaces. We therefore first need rough estimates for the evolution on $H^1 \oplus L^2$.

\begin{lemma}\sl \label{b2}
Uniformly for $t \in \R$, we have
\begin{gather}
\big\Vert e^{- i t G_{0}} \big\Vert_{\CL (H^1\oplus L^2)} \lesssim \<t\> , \label{b3}    \\
\big\Vert e^{- i t G_{0}} \< x \>^{-1} \big\Vert_{\CL (H^1\oplus L^2)} \lesssim 1 .   \label{b4}
\end{gather}
\end{lemma}

\begin{proof}
Using the functional calculus, we obtain
\begin{gather*}
\big\Vert \cos t \sqrt{P_0} \big\Vert_{{\mathcal L}(H^1)} \lesssim 1 , \qquad \big\Vert \cos t \sqrt{P_0} \big\Vert_{{\mathcal L} (L^2)} \lesssim 1 ,   \\
\big\Vert - \sqrt{P_0} \sin t \sqrt{P_0} \big\Vert_{{\mathcal L} (H^1 , L^2)} = \big\Vert - \sqrt{P_0} \sin ( t \sqrt{P_0} ) \< P_0 \>^{-1/2} \big\Vert_{{\mathcal L} (L^2)} \lesssim 1 ,  \\
\begin{aligned}
\Big\Vert \frac{\sin t\sqrt{P_0}}{\sqrt{P_0}} \Big\Vert_{{\mathcal L}(L^2,H^1)} ={}& \Big\Vert \<P_0\>^{1/2}\frac{\sin t\sqrt{P_0}}{\sqrt{P_0}} \Big\Vert_{{\mathcal L}(L^2)}   \\
\lesssim{}& \Big\Vert \<P_0\>^{1/2} t \frac{\sin t \sqrt{P_0}}{t \sqrt{P_0}} \chi( P_0 \leq 1) \Big\Vert_{{\mathcal L}(L^2)}   \\
&+ \Big\Vert \< P_0 \>^{1/2} \frac{\sin t \sqrt{P_0}}{\sqrt{P_0}} \chi ( P_0 \geq 1 ) \Big\Vert_{{\mathcal L}(L^2)} \lesssim \< t \> .
\end{aligned}
\end{gather*}
Combined with \eqref{b5}, this implies the first estimate.

We now prove \eqref{b4}. Using the previous arguments, we only have to show
\begin{gather*}
\Big\Vert \< P_0 \>^{1/2} \frac{\sin t \sqrt{P_0}}{\sqrt{P_0}} \chi ( P_0 \leq 1 ) \< x \>^{- 1} \Big\Vert_{\CL {(L^2)}} \lesssim 1 .
\end{gather*}
By the classical Hardy estimate, we have
\begin{equation*}
\Big\Vert \frac{1}{\vert x \vert} u \Big\Vert \lesssim \big\Vert \vert \nabla \vert u \big\Vert ,
\end{equation*}
which gives by Fourier transform 
\begin{equation*}
\Big\Vert \frac{1}{\sqrt{P_0}} u \Big\Vert \lesssim \big\Vert \vert x \vert u \big\Vert \lesssim \big\Vert \< x \> u \big\Vert .
\end{equation*}
We conclude that
\begin{equation*}
\Big\Vert \< P_0 \>^{1/2} \frac{\sin t \sqrt{P_0}}{\sqrt{P_0}} \chi ( P_0 \leq 1 ) \< x \>^{- 1} \Big\Vert \lesssim \Big\Vert \frac{1}{\sqrt{P_0}} \< x \>^{- 1} \Big\Vert \lesssim 1 ,
\end{equation*}
and the second estimate of the lemma follows.
\end{proof}

The following estimate on the free evolution is fundamental for the proof of Proposition \ref{b1}.

\begin{lemma}\sl \label{b6}
Let $d \geq 3$ be odd and $\alpha \geq 1$. Then, uniformly in $t \in \R$, we have
\begin{equation*}
\big\Vert \< x \>^{- \alpha} e^{- i t G_{0}} \< x \>^{-\alpha} \big\Vert_{\CL ( H^1 \oplus L^2 )} \lesssim \< t \>^{- \alpha} .
\end{equation*}
\end{lemma}

\begin{proof}
By \eqref{b3}, we can suppose $t \geq 1$. Let $\varphi \in C_0^{\infty} ( ] - \infty , \frac{1}{2} [ )$ be such that $\varphi = 1$ close to zero. In particular, this implies
\begin{equation} \label{b7}
\Big\Vert \varphi \Big( \frac{\vert x \vert}{t} \Big) \Big\Vert_{\CL ( H^1 \oplus L^2 )} \lesssim 1 ,
\end{equation}
and
\begin{equation*}
\Big\Vert \< x \>^{- \alpha} (1 - \varphi ) \Big( \frac{\vert x \vert}{t} \Big) \Big\Vert_{\CL ( H^1 \oplus L^2 )} \leq \Big\Vert \< x \>^{- \alpha} (1 - \varphi ) \Big( \frac{\vert x \vert}{t} \Big) \Big\Vert_{\CL ( L^2 )} + \Big\Vert \< x \>^{- \alpha} (1 - \varphi ) \Big( \frac{\vert x \vert}{t} \Big) \Big\Vert_{\CL ( H^1 )} .
\end{equation*}
We obviously have
\begin{equation*}
\Big\Vert \< x \>^{- \alpha} (1 - \varphi ) \Big( \frac{\vert x \vert}{t} \Big) \Big\Vert_{\CL ( L^2 )} \lesssim \< t \>^{- \alpha} .
\end{equation*}
On the other hand,
\begin{align*}
\Big\Vert \< x \>^{- \alpha} ( 1 - \varphi ) \Big( \frac{\vert x \vert}{t} \Big) u \Big\Vert_{H^1} \lesssim{}& \Big\Vert \< x \>^{- \alpha} ( 1 - \varphi ) \Big( \frac{\vert x \vert}{t} \Big) u \Big\Vert_{L^2} + \Big\Vert \< x \>^{- \alpha - 1} ( 1 - \varphi ) \Big( \frac{\vert x \vert}{t} \Big) u \Big\Vert_{L^2}   \\
&+ \Big\Vert \< x \>^{- \alpha} \varphi^{\prime} \Big( \frac{\vert x \vert}{t} \Big) \frac{x}{t \vert x \vert} u \Big\Vert_{L^2} + \Big\Vert \< x \>^{- \alpha} ( 1 - \varphi ) \Big( \frac{\vert x \vert}{t} \Big) \nabla u \Big\Vert_{L^2}    \\
\lesssim{}& \< t \>^{- \alpha} \Vert u \Vert_{H^1} .
\end{align*}
Combining the three previous estimates, it yields
\begin{equation} \label{b8}
\Big\Vert \< x \>^{- \alpha} (1 - \varphi ) \Big( \frac{\vert x \vert}{t} \Big) \Big\Vert_{\CL ( H^1 \oplus L^2 )} \lesssim \< t \>^{- \alpha} .
\end{equation}

Using the previous estimates, we can now finish the proof of the lemma. We write
\begin{align}
\big\Vert \< x \>^{- \alpha} e^{- i t G_{0}} \< x \>^{- \alpha} \big\Vert_{\CL ( H^1 \oplus L^2 )} \leq{}& \Big\Vert \< x \>^{- \alpha} ( 1 - \varphi ) \Big( \frac{\vert x \vert}{t} \Big) e^{- i t G_{0}} \< x \>^{- \alpha} \Big\Vert_{\CL ( H^1 \oplus L^2)}    \nonumber \\ 
&+ \Big\Vert \< x \>^{- \alpha} \varphi \Big( \frac{\vert x \vert}{t} \Big) e^{- i t G_{0}} ( 1 - \varphi ) \Big ( \frac{\vert x \vert}{t} \Big) \< x \>^{- \alpha} \Big\Vert_{\CL ( H^1 \oplus L^2)}   \nonumber \\
&+ \Big\Vert \< x \>^{- \alpha} \varphi \Big( \frac{\vert x \vert}{t} \Big) e^{- i t G_{0}} \varphi \Big( \frac{\vert x \vert}{t} \Big) \< x \>^{- \alpha} \Big\Vert_{\CL ( H^1 \oplus L^2 )}     \nonumber \\
=&: I_1 + I_2 + I_3 .  \label{b9}
\end{align}
Using \eqref{b4}, \eqref{b8} and $\alpha \geq 1$, we get
\begin{equation} \label{b10}
I_1 \lesssim \Big\Vert \< x \>^{- \alpha} ( 1 - \varphi ) \Big( \frac{\vert x \vert}{t} \Big) \Big\Vert_{\CL ( H^1 \oplus L^2)} \Big\Vert e^{- i t G_{0}} \< x \>^{- \alpha} \Big\Vert_{\CL ( H^1 \oplus L^2)} \lesssim \< t \>^{- \alpha} .
\end{equation}
The same way, \eqref{b4}, \eqref{b7}, \eqref{b8} and $\alpha \geq 1$ imply
\begin{align}
I_2 &\lesssim \Big\Vert \varphi \Big( \frac{\vert x \vert}{t} \Big) \Big\Vert_{\CL ( H^1 \oplus L^2)} \Big\Vert \< x \>^{- \alpha} e^{- i t G_{0}} \Big\Vert_{\CL ( H^1 \oplus L^2)} \Big\Vert ( 1 - \varphi ) \Big ( \frac{\vert x \vert}{t} \Big) \< x \>^{- \alpha} \Big\Vert_{\CL ( H^1 \oplus L^2)}    \nonumber \\
&\lesssim \< t \>^{- \alpha} .   \label{b11}
\end{align}
Eventually, the strong Huygens principle and the assumptions on the support of $\varphi$ give
\begin{equation} \label{b12}
I_{3} = 0 .
\end{equation}
Thus, the lemma follows from \eqref{b9} and the estimates \eqref{b10}, \eqref{b11} and \eqref{b12}.
\end{proof}

\begin{proof}[Proof of Proposition \ref{b1}]
We have, for $\im z >0$,
\begin{equation*}
( G_{0} - z )^{-1} = i \int_0^{+ \infty} e^{- i t (G_0-z)} d t ,
\end{equation*}
and thus
\begin{equation*}
( G_{0} - z )^{-k} = i \int_0^{+ \infty} \frac{(i t)^{k-1}}{( k-1 ) !} e^{- i t (G_0-z)} d t .
\end{equation*}
We then estimate, for $\im z > 0$,
\begin{align*}
\big\Vert \< x \>^{- k - \varepsilon} ( G_{0} - z )^{-k} \< x \>^{- k - \varepsilon} \big\Vert_{{\CL}({\CH}^0)} &\leq \int_0^{+ \infty} \vert t \vert^{k-1} \big\Vert \< x \>^{- k - \varepsilon} e^{- i t G_0} \< x \>^{- k - \varepsilon} \big\Vert_{{\CL}({\CH}^0)} d t    \\
&\lesssim \int_0^{+ \infty} \< t \>^{- 1 - \varepsilon} d t \lesssim 1 ,
\end{align*}
where we have used Lemma \ref{b6}. To obtain the higher order estimates, we observe that
\begin{equation*}
\Vert v \Vert_{\CH^{\beta}} \asymp \big\Vert (G_0+i)^{\beta} v \big\Vert_{\CH^{0}} ,
\end{equation*}
and that
\begin{equation*}
(G_0+i)^{\beta} \< x \>^{-k-\varepsilon} = \CO (1) \<x\>^{-k-\varepsilon} (G_0+i)^{\beta} .
\end{equation*}
The proof for $\im z < 0$ is analogous.
\end{proof}

\section{Improved estimates for the free resolvent in dimension 3}  \label{s3}

In this section, we show improved resolvent estimates in dimension $3$ using the explicit form of the kernel.

\begin{proposition}\sl \label{b13}
Let $d = 3$, $k \in \N^{*}$, $s \in \R$ and
\begin{equation*}
\kappa =
\left\{ \begin{aligned}
&1 &&\text{for } k = 1 ,   \\
&k - 1/2 &&\text{for } k \geq 2 .
\end{aligned} \right.
\end{equation*}
Then, for all $C , \varepsilon > 0$, we have
\begin{equation*}
\sup_{z \in \C \setminus \R , \ \vert z \vert \leq C} \big\Vert \< x \>^{- \kappa - \varepsilon} ( G_{0} - z )^{-k} \< x \>^{- \kappa - \varepsilon} \big\Vert_{\CL ( \CH^s , \CH^{s+k} )} \lesssim 1 .
\end{equation*}
\end{proposition}

In order to prove this proposition, we will need the following lemma valid in all dimensions.

\begin{lemma}\sl \label{b14}
Let $0 < \alpha , \gamma < \frac{d}{2}$, $0 < \beta < d$ be such that $\frac{d}{2} < \alpha + \beta < d$ and $\alpha + \beta + \gamma > d$. Then the operator with integral kernel $k ( x , y ) = \< x \>^{- \alpha} \vert x - y \vert^{- \beta} \< y \>^{- \gamma}$ is a bounded operator on $L^2 ( \R^d )$.
\end{lemma}

\begin{proof}
Let $u \in L^2 ( \R^d )$ and 
\begin{equation*}
v ( x ) = \int \< x \>^{- \alpha} \vert x - y \vert^{- \beta} \< y \>^{- \gamma} u ( y ) \, d y .
\end{equation*}
Then, by the H\"older inequality, we have
\begin{equation*}
\Vert v \Vert \lesssim \big\Vert \< x \>^{- \alpha + \varepsilon} \big\Vert_{p_1} \Big\Vert \int \vert x - y \vert^{- \beta} \< x - y \>^{- \varepsilon} \< y \>^{- \gamma + \varepsilon} \vert u \vert ( y ) \, d y \Big\Vert_{p_2} ,
\end{equation*}
with $\frac{1}{p_1} + \frac{1}{p_2} = \frac{1}{2}$. For $0 < \alpha < \frac{d}{2}$, we can take $p_1 = \frac{d + \varepsilon}{\alpha - \varepsilon}$ and $p_2 = \frac{2 d + 2 \varepsilon}{d - 2 \alpha + 3 \varepsilon}$ with $\varepsilon > 0$ small enough. Now, by \cite[Corollary 4.5.2]{Ho83_01}, we have
\begin{equation*}
\big\Vert \vert x \vert^{- \beta} \< x \>^{- \varepsilon} * \< y \>^{- \gamma + \varepsilon} \vert u \vert ( y ) \big\Vert_{p_2} \lesssim \big\Vert \vert x \vert^{- \beta} \< x \>^{- \varepsilon} \big\Vert_{q_1} \big\Vert \< y \>^{- \gamma + \varepsilon} \vert u \vert \big\Vert_{q_2} ,
\end{equation*}
with $1 \leq q_{1} , q_{2} \leq + \infty$ and
\begin{equation*}
\frac{1}{q_1} + \frac{1}{q_2} = 2 - 1 + \frac{1}{p_2} .
\end{equation*}
As $0 < \beta < d$ and $\alpha + \beta > d / 2$, we can take $q_1 = \frac{d + \varepsilon}{\beta + \varepsilon}$ and $q_2 = \frac{2 d + 2 \varepsilon}{3 d - 2 ( \alpha + \beta ) + 3 \varepsilon}$. We now estimate again by the H\"older inequality
\begin{equation*}
\big\Vert \< y \>^{- \gamma + \varepsilon} \vert u \vert \big\Vert_{q_2} \lesssim \big\Vert \< y \>^{- \gamma + \varepsilon} \big\Vert_{r_1} \Vert u \Vert_{r_2} ,
\end{equation*}
with
\begin{equation*}
\frac{1}{r_1} + \frac{1}{r_2} = \frac{1}{q_2} .
\end{equation*}
As $\alpha + \beta < d$, we can take $r_2 = 2$ and $\frac{1}{r_1} = 1 - \frac{\alpha + \beta}{d + \varepsilon}$. We need $( \gamma - \varepsilon ) r_1 > d$ or equivalently
\begin{equation*}
\gamma - \varepsilon > \frac{d}{r_1} = d - \frac{d ( \alpha + \beta )}{d + \varepsilon} \Longleftrightarrow d ( \alpha + \beta + \gamma ) > d^2 + 2 \varepsilon d - \varepsilon \gamma + \varepsilon^{2} ,
\end{equation*}
which is fulfilled for $\varepsilon > 0$ small enough since $\alpha + \beta + \gamma > d$.
\end{proof}

\begin{proof}[Proof of Proposition \ref{b13}]
Using Proposition \ref{b1}, it is sufficient to consider the case $k \geq 2$. Let us first recall that the kernel of the free resolvent in dimension $3$ is given by
\begin{equation*}
( P_0 - z^2 )^{-1} \delta = \frac{e^{i z \vert x\vert}}{4 \pi \vert x \vert} .
\end{equation*}
Using \eqref{c27}, we write
\begin{align*}
( G_{0} - z )^{-1} &= ( P_0 - z^2 )^{-1}
\left( \begin{array}{cc}
z & i 
\\ - i P_0 & z
\end{array} \right)  \\
&=
\left( \begin{array}{cc}
0 & 0 \\
- i & 0
\end{array} \right)
+ ( P_0 - z^2 )^{-1}
\left( \begin{array}{cc}
z & i \\
- i z^2 & z
\end{array} \right) .
\end{align*}
The kernel of the second operator in the above line is given by 
\begin{equation*}
\left(\begin{array}{cc}
z & i \\ - i z^2 & z
\end{array} \right)
\frac{1}{4 \pi \vert x - y \vert}e^{i z \vert x - y \vert} .
\end{equation*}
Note also that, for $k \geq 2$,
\begin{equation*}
( G_{0} - z )^{-k} = \frac{1}{( k - 1 ) !} \partial_z^{k - 1} ( G_0 - z )^{-1} .
\end{equation*}
Thus, the kernel of this operator decomposes into a sum of terms of the form
\begin{equation} \label{b15}
z^{\beta} \vert x - y \vert^{\gamma - 1} e^{i z \vert x - y \vert} ,
\end{equation}
with $0 \leq \beta \leq 2$, $0 \leq \gamma \leq k - 1$. We therefore have to bound kernels of the form
\begin{equation*}
\< x \>^{- \kappa - \varepsilon} \vert x - y \vert^{\gamma - 1} e^{i z \vert x - y \vert} \< y \>^{- \kappa - \varepsilon} .
\end{equation*}
If $\gamma - 1 = - 1$, then Lemma \ref{b14} tells us that the corresponding operator is bounded on $L^2$ for $\kappa \geq 1$. If $\gamma - 1 \geq 0$, then for $\kappa = \gamma - 1 + 3 / 2$ the above kernel can be estimated by 
\begin{equation*}
\< x \>^{- 3 / 2 - \varepsilon} \< y \>^{- 3 / 2 - \varepsilon} ,
\end{equation*}
which clearly defines a bounded operator on $L^2$. The worst case is $\gamma = k - 1$ and thus $\kappa = k - 1 / 2$. This proves the estimate in $\CL ( \CH^0 )$ and then in $\CL ( \CH^s , \CH^{s+k} )$ by the same argument as in the end of the proof of Proposition \ref{b1}. 
\end{proof}

\section{Resolvent estimates for the perturbed operator} \label{s4}

Using the results obtained in the previous sections, we now prove the estimates for the weighted resolvent of $G$ stated in Theorem \ref{b16}. To lighten the exposition, we will use the notations $R(z)=(G-z)^{-1}$ and $R_0(z)=(G_0-z)^{-1}$ in the sequel. Let $\widetilde{\partial}_j = \partial_j b$ and $\widetilde{\partial}_j^{*} = b \partial_j$. In the following, $r_j$ will stand for an error term fulfilling
\begin{equation} \label{c3}
\partial^{\alpha}_x r_j ( x ) = \CO \big( \< x \>^{- \vert \alpha \vert - \rho - j} \big) .
\end{equation} 
Let us now introduce 
\begin{equation} \label{c11}
V : = G_{0} - G = i \left( \begin{array}{cc} 0 & 0 \\
P - P_0 & 0\end{array}\right) : {\mathcal H}^s \longrightarrow {\mathcal H}^{s - 1} ,
\end{equation}
which is continuous. Note that 
\begin{equation} \label{c14}
P - P_0 = \widetilde{\partial}^* r_0 \widetilde{\partial} + \widetilde{\partial}^* r_1 + r_1 \widetilde{\partial} + r_2 ,
\end{equation}
where we have not written the sum over the indexes on the right hand side. In dimension $3$, we will need the following lemma. Note that this result also applies to $G$ replaced by $G_{0}$.

\begin{lemma}\sl \label{c2}
For all $s \in \R$ and $C , \varepsilon > 0$, we have
\begin{gather}
\bigg\Vert \< x \>^{- 1 / 2 - \varepsilon} \left( \begin{array}{cc} 0 & 0 \\ \widetilde{\partial}_j & 0 \end{array} \right) R ( z ) \< x \>^{- 1 / 2 - \varepsilon} \bigg\Vert_{\CL ( \CH^s , \CH^{s + 1} )} \lesssim 1 ,  \label{c4}  \\
\bigg\Vert \< x \>^{- 1 / 2 - \varepsilon} R ( z ) \left( \begin{array}{cc} 0 & 0 \\ \widetilde{\partial}^*_j & 0 \end{array} \right) \< x \>^{- 1 / 2 - \varepsilon} \bigg\Vert_{\CL ( \CH^s , \CH^{s + 1} )} \lesssim 1 ,   \label{c5}
\end{gather}
uniformly in $z \in \C \setminus \R$, $\vert z \vert \leq C$.
\end{lemma}

\begin{proof}
We only show \eqref{c4}, the proof for \eqref{c5} being analogous. Let us first recall that 
\begin{gather}
G = U^{-1} L U \qquad \text{with} \qquad L = \left( \begin{array}{cc} P^{1/2} & 0 \\
0 & -P^{1/2} \end{array} \right) ,    \label{c6} \\
U = \frac{1}{\sqrt{2}} \left( \begin{array}{cc} P^{1/2} & i \\
P^{1/2} & -i \end{array}\right) , \qquad U^{-1} = \frac{1}{\sqrt{2}} \left( \begin{array}{cc} P^{-1/2} & P^{-1/2} \\
-i & i \end{array} \right) .     \label{c7}
\end{gather}
Therefore 
\begin{equation} \label{c8}
R ( z ) = U^{-1} ( L - z )^{-1} U .
\end{equation}
Note that $U : \dot{H}^1_P \oplus L^2 \rightarrow L^2 \oplus L^2$ is a unitary transform and that $L$ is selfadjoint on $L^2\oplus L^2$ with domain $D (L) = H^1 \oplus H^1$. Using \eqref{c6}--\eqref{c8}, we compute
\begin{equation*}
\left(\begin{array}{cc} 0 & 0 \\ \widetilde{\partial}_j & 0 \end{array} \right) R ( z )
=\frac{1}{2}\left(\begin{array}{cc} 0 & 0\\ A & B\end{array} \right) ,
\end{equation*} 
with
\begin{gather*}
A = \widetilde{\partial}_j ( P^{1 / 2} - z )^{-1} + \widetilde{\partial}_j ( - P^{1 / 2} - z )^{-1} ,  \\ 
B = i \big( \widetilde{\partial}_j ( P^{1 / 2} - z )^{-1} - \widetilde{\partial}_j ( - P^{1/2} - z )^{-1} \big) P^{- 1 / 2} .
\end{gather*}
In order to prove a bound on $\CL ( \CH^0 , \CH^1 )$, it is therefore sufficient to show that 
\begin{gather*}
\< x \>^{- 1 / 2 - \varepsilon} \widetilde{\partial}_j ( P^{1 / 2} - z )^{-1} \< x \>^{- 1 / 2 - \varepsilon} : H^{1} \longrightarrow H^{1} ,   \\
\< x \>^{- 1 / 2 - \varepsilon}\widetilde{\partial}_j P^{-1/2}(P^{1/2}-z)^{-1}\<x\>^{-1/2-\varepsilon} : L^{2} \longrightarrow H^{1} ,
\end{gather*}
are bounded uniformly in $z \in \C \setminus \R$, $\vert z \vert \leq C$. This follows from \cite[Lemmas 4.1, 4.7 and 4.8]{BoHa10_01} and \cite[Theorem 1]{BoHa10_02}. In order to prove the estimates for $s \in \Z$, we commute with the partial derivatives $\widetilde{\partial}_k$ and use \cite[Lemmas 4.1, 4.7 and 4.8]{BoHa10_01}. Eventually, the case $s \in \R$ follows from an interpolation argument.
\end{proof}

To prove Theorem \ref{b16}, it will be useful to have an explicit form of the powers of the perturbed resolvent $R^{k} ( z )$ in terms of the powers of the free resolvent $R_{0}^{j} ( z )$, $1 \leq j \leq k$, and of the perturbed resolvent $R ( z )$.

\begin{lemma}\sl \label{c1}
For all $k \in \N^{*}$ and $z \in \C \setminus \R$, we have
\begin{equation*}
R^k(z) = \sum_{\text{\rm{finite}}} M_0 V \cdots V M_{n} ,
\end{equation*}
with $M_0 = R_0^{\alpha_0} ( z )$, $M_n = R_0^{\alpha_n} ( z )$ and $M_j = R ( z )$ (in which case we put $\alpha_j = 1$) or $M_j = R_0^{\alpha_j} ( z )$ for $1 \leq j \leq n-1$. Moreover, the $\alpha_{j}$'s satisfy
\begin{equation*}
\forall j \in \{ 0 , \ldots , n \} \qquad 0 < \alpha_j \leq k , \quad \alpha_j + \alpha_{j+1} \leq k + 1 \qquad \text{and} \qquad \sum_{j=0}^{n} \alpha_j = n + k .
\end{equation*}
\end{lemma}

\begin{proof}
We prove the lemma by induction over $k$. In the case $k=1$, we use twice the resolvent identity:
\begin{align*}
R (z) &= R_0(z) + R_0 (z) V R (z)  \\
&= R_0 (z) + R_0 (z) V R_0 (z) + R_0 (z) V R(z) V R_0 (z) .
\end{align*}
Let us now suppose the lemma for $k \geq 1$. We write
\begin{align*}
R^{k + 1} (z) &= R (z) R^k (z)  \\
&= \big( R_0 (z) + R_0 (z) V R_0 (z) + R_0 (z) V R (z) V R_0 (z) \big) \sum_{\text{\rm{finite}}} M_0 V \cdots V M_n   \\
&= \sum_{\text{\rm{finite}}} M_0 V \cdots V M_{m} ,
\end{align*}
where the last sum has the required properties.
\end{proof}

\begin{proof}[Proof of Theorem \ref{b16}]
Let us first consider the case $k=1$. From \eqref{c9}, we have
\begin{equation} \label{c10}
R ( z ) = ( P - z^2 )^{-1} \left( \begin{array}{cc} z & i \\
-i P & z \end{array} \right) .
\end{equation}
Using \cite[Theorem 1]{BoHa10_02} and a simple calculation, we get
\begin{equation*}
\big\Vert \< x \>^{- 1 - \varepsilon} ( P - z^2 )^{-1} \< x \>^{- 1 - \varepsilon} \big\Vert_{\CL ( H^s , H^{s + 2} )} \lesssim 1 ,
\end{equation*}
uniformly in $z \in \C \setminus \R$, $\vert z \vert\leq C$. It then follows by \eqref{c10} that 
\begin{equation} \label{c13}
\big\Vert \< x \>^{- 1 - \varepsilon} R ( z ) \< x \>^{- 1 - \varepsilon} \big\Vert_{\CL ( \CH^s , \CH^{s + 1 } )} \lesssim 1 ,
\end{equation}
and the case $k = 1$ follows.

We now treat the case $k \geq 2$, $d \geq 3$ odd. Using Lemma \ref{c1}, we can write
\begin{align}
\< x \>^{- k - \varepsilon} R^{k} (z) \< x \>^{- k - \varepsilon} &= \sum_{\text{\rm{finite}}} \< x \>^{- k - \varepsilon} M_0 V \cdots M_{j} V M_{j + 1} \cdots V M_{n} \< x \>^{- k - \varepsilon}   \nonumber  \\
&= \sum_{\text{\rm{finite}}} \< x \>^{\alpha_{0} - k} \< x \>^{- \alpha_{0} - \varepsilon} M_0 \< x \>^{- \alpha_{0} - \varepsilon} \< x \>^{\alpha_{0} + \varepsilon} V \cdots    \nonumber   \\
&\qquad \cdots M_{j} \< x \>^{- \alpha_{j} - \varepsilon} \< x \>^{\alpha_{j} + \varepsilon} V \< x \>^{\alpha_{j + 1} + \varepsilon}  \< x \>^{- \alpha_{j + 1} - \varepsilon} M_{j + 1} \cdots   \nonumber   \\
&\qquad \cdots V \< x \>^{\alpha_{n} + \varepsilon} \< x \>^{- \alpha_{n} - \varepsilon} M_{n} \< x \>^{- \alpha_{n} - \varepsilon} \< x \>^{\alpha_{n} - k} .  \label{c12}
\end{align}
Since $\alpha_{j} + \alpha_{j + 1} \leq k + 1 < \rho$, \eqref{c11} and \eqref{c14} imply that
\begin{equation*}
\< x \>^{\alpha_{j} + \varepsilon} V \< x \>^{\alpha_{j + 1} + \varepsilon} : {\mathcal H}^s \longrightarrow {\mathcal H}^{s - 1} ,
\end{equation*}
is a bounded operator. Moreover, from Proposition \ref{b1} and \eqref{c13}, we have
\begin{equation*}
\big\Vert \< x \>^{- \alpha_{j} - \varepsilon} M_{j} \< x \>^{- \alpha_{j} - \varepsilon} \big\Vert_{\CL ( \CH^s , \CH^{s + \alpha_{j} } )} \lesssim 1 ,
\end{equation*}
uniformly in $z \in \C \setminus \R$, $\vert z \vert \leq C$. Combining \eqref{c12} with the previous estimates, $\alpha_{0} \leq k$, $\alpha_{n} \leq k$ and $\sum \alpha_{j} = k + n$, we get that $\< x \>^{- k - \varepsilon} R^{k} (z) \< x \>^{- k - \varepsilon}$ is bounded uniformly in $z \in \C \setminus \R$, $\vert z \vert \leq C$ as operator from $\CH^{s}$ to $\CH^{s + k}$.

It remains to study the case $k \geq 2$ and $d = 3$. As before, Lemma \ref{c1} gives
\begin{align}
\< x \>^{- k + 1 / 2 - \varepsilon} R^{k} (z & ) \< x \>^{- k + 1 / 2 - \varepsilon}  \nonumber  \\
&= \sum_{\text{\rm{finite}}} \< x \>^{- k + 1 / 2 - \varepsilon} M_0 V \cdots M_{j} V M_{j + 1} \cdots V M_{n} \< x \>^{- k + 1 / 2 - \varepsilon} . \label{c15}
\end{align}
From \eqref{c11} and \eqref{c14}, we have
\begin{align*}
V ={}& \left( \begin{array}{cc} 0 & 0 \\ \widetilde{\partial}^{*} & 0 \end{array} \right) \left( \begin{array}{cc} 0 & i r_{0} \\ 0 & 0 \end{array} \right) \left( \begin{array}{cc} 0 & 0 \\ \widetilde{\partial} & 0 \end{array} \right) + \left( \begin{array}{cc} 0 & 0 \\ \widetilde{\partial}^{*} & 0 \end{array} \right) \left( \begin{array}{cc} i r_{1} & 0 \\ 0 & 0 \end{array} \right)   \\
&+ \left( \begin{array}{cc} 0 & 0 \\ 0 & i r_{1} \end{array} \right) \left( \begin{array}{cc} 0 & 0 \\ \widetilde{\partial} & 0 \end{array} \right) + \left( \begin{array}{cc} 0 & 0 \\ i r_{2} & 0 \end{array} \right) .
\end{align*}
where the sum over the indexes does not appear. In particular, since $\alpha_{j} + \alpha_{j + 1} \leq k + 1 < \rho + 1$, $V$ can be written as
\begin{equation} \label{c16}
V = \sum_{\text{\rm{finite}}} A_{j}^{*} B A_{j + 1} ,
\end{equation}
where
\begin{gather*}
A_{j} = \< x \>^{- \alpha_{j} + 1 / 2 - \varepsilon} \left( \begin{array}{cc} 0 & 0 \\ \widetilde{\partial} & 0 \end{array} \right) \qquad \text{or} \qquad A_{j} = \< x \>^{- \alpha_{j} - \varepsilon} ,    \\
A_{j}^{*} = \left( \begin{array}{cc} 0 & 0 \\ \widetilde{\partial}^{*} & 0 \end{array} \right) \< x \>^{- \alpha_{j} + 1 / 2 - \varepsilon} \qquad \text{or} \qquad A_{j}^{*} = \< x \>^{- \alpha_{j} - \varepsilon} ,
\end{gather*}
and $B$ is a bounded operator $\CH^{s}$ to $\CH^{s - 1}$. From Proposition \ref{b13}, Lemma \ref{c2} and \eqref{c13}, we have, for $j \in \{ 1 , \ldots n-1 \}$,
\begin{equation} \label{c17}
\big\Vert A_{j} M_{j} A_{j}^{*} \big\Vert_{\CL ( \CH^{s} , \CH^{s + \alpha_{j}} )} \lesssim 1 ,
\end{equation}
uniformly in $z \in \C \setminus \R$, $\vert z \vert \leq C$. Moreover, since $k \geq 2$, we have $k - 1/2 \geq \max ( 1 , \alpha_{0} - 1 / 2 )$ and $k - 1/2 \geq \max ( 1 , \alpha_{n} - 1 / 2 )$. Then, Proposition \ref{b13} and Lemma \ref{c2} give
\begin{equation} \label{c18}
\big\Vert \< x \>^{- k + 1 / 2 - \varepsilon} M_{0} A_{0}^{*} \big\Vert_{\CL ( \CH^{s} , \CH^{s + \alpha_{0}} )} \lesssim 1 \quad \text{and} \quad \big\Vert A_{n} M_{n} \< x \>^{- k + 1 / 2 - \varepsilon} \big\Vert_{\CL ( \CH^{s} , \CH^{s + \alpha_{n}} )} \lesssim 1 .
\end{equation}
Putting together \eqref{c15}, \eqref{c16}, \eqref{c17} and \eqref{c18}, we get that $\< x \>^{- k + 1 / 2 - \varepsilon} R^{k} (z) \< x \>^{- k + 1 / 2 - \varepsilon}$ is bounded uniformly in $z \in \C \setminus \R$, $\vert z \vert \leq C$ as operator from $\CH^{s}$ to $\CH^{s + k}$.
\end{proof}

It turns out that the weighted resolvent of $G$ is not only bounded, but also has some H\"older regularity which will be used in the proof of Theorem \ref{b17}.

\begin{proposition}\sl \label{c19}
Assume $d \geq 3$ odd, $k \in \N$ with $k \geq 2$, $\alpha \in ] 0 , 1 [$ and $\rho > k + \alpha + 1$ ($\rho > k + \alpha$ for $d = 3$). Let $\kappa=k$ ($\kappa = k - 1 / 2$ for $d=3$). Then, for all $s \in \R$ and $C , \varepsilon>0$, we have
\begin{equation*}
\big\Vert \< x \>^{- \kappa - \alpha - \varepsilon} \big( R^k ( z ) - R^k ( z^{\prime} ) \big) \< x \>^{- \kappa - \alpha - \varepsilon} \big\Vert_{\CL ( \CH^s, \CH^{s + k} )} \lesssim \vert z - z^{\prime} \vert^{\alpha} ,
\end{equation*}
uniformly in $\vert z \vert , \vert z^{\prime} \vert \leq C$ with $\im z \cdot \im z^{\prime} > 0$.
\end{proposition}

\begin{corollary}\sl \label{c26}
Proposition \ref{c19} and a classical argument imply that the powers of the weighted resolvent have a limit on the real axis. More precisely, under the assumptions of Proposition \ref{c19}, the limits
\begin{equation*}
\< x \>^{- \kappa - \varepsilon} R^{j} ( \lambda \pm i 0 ) \< x \>^{- \kappa - \varepsilon} = \lim_{\delta \downarrow 0} \< x \>^{- \kappa - \varepsilon} R^{j} ( \lambda \pm i \delta ) \< x \>^{- \kappa - \varepsilon} ,
\end{equation*}
exist for $\lambda \in ] - C , C [$ and $j \in \{ 1 , \ldots , k \}$. Moreover, for $j \in \{ 1 , \ldots , k -1 \}$,
\begin{equation*}
\< x \>^{- \kappa - \varepsilon} R^{j + 1} ( \lambda \pm i 0 ) \< x \>^{- \kappa - \varepsilon} = j^{- 1} \partial_{\lambda} \< x \>^{- \kappa - \varepsilon} R^{j} ( \lambda \pm i 0) \< x \>^{- \kappa - \varepsilon} ,
\end{equation*}
and
\begin{equation*}
\big\Vert \< x \>^{- \kappa - \alpha - \varepsilon} \big( R^k ( \lambda \pm i 0 ) - R^k ( \lambda^{\prime} \pm i 0 ) \big) \< x \>^{- \kappa - \alpha - \varepsilon} \big\Vert_{\CL ( \CH^s, \CH^{s + k} )} \lesssim \vert \lambda - \lambda^{\prime} \vert^{\alpha} ,
\end{equation*}
uniformly in $\lambda , \lambda^{\prime} \in ] - C , C [$.
\end{corollary}

\begin{proof}[Proof of Proposition \ref{c19}]
An interpolation argument using Proposition \ref{b1} and Proposition \ref{b13} gives, for $j \geq 2$,
\begin{equation} \label{c20}
\Big\Vert \< x \>^{- \kappa - \alpha - \varepsilon} \big( R^j_0 ( z ) - R^j_0 ( z^{\prime} ) \big) \< x \>^{- \kappa - \alpha - \varepsilon} \big\Vert_{\CL ( \CH^s, \CH^{s + j} )} \lesssim \vert z - z^{\prime} \vert^{\alpha} ,
\end{equation}
with
\begin{equation*}
\kappa =
\left\{ \begin{aligned}
&j &&\text{for } d \geq 5 ,   \\
&j - 1/2 &&\text{for } d = 3 .
\end{aligned} \right.
\end{equation*}
Since $k \geq 2$, Theorem \ref{b16} yields
\begin{equation*}
\big\Vert \< x \>^{- 1 - \varepsilon} R ( z ) \< x \>^{- 1 - \varepsilon} \big\Vert_{\CL ( \CH^s, \CH^{s + 1} )} \lesssim 1 , \quad \big\Vert \< x \>^{- 2 - \varepsilon} R^2 ( z ) \< x \>^{- 2 - \varepsilon} \big\Vert_{\CL ( \CH^s, \CH^{s + 2} )} \lesssim 1 ,
\end{equation*}
for all $d \geq 3$ odd. This gives
\begin{equation} \label{c21}
\big\Vert \< x \>^{- 1 - \alpha - \varepsilon} \big( R ( z ) - R( z^{\prime} ) \big) \< x \>^{- 1 - \alpha - \varepsilon} \big\Vert_{\CL ( \CH^s, \CH^{s + 1} )} \lesssim \vert z- z^{\prime} \vert^{\alpha} .
\end{equation}

For the improvement in dimension $d = 3$, we need estimates in the spirit of Lemma \ref{c2}. Since $k \geq 2$, Theorem \ref{b16} yields\begin{equation*}
\big\Vert \< x \>^{- 3 / 2 - \varepsilon} R^2 ( z ) \< x \>^{- 3 / 2 - \varepsilon} \big\Vert_{\CL ( \CH^s, \CH^{s + 2} )} \lesssim 1 ,
\end{equation*}
and then
\begin{equation*}
\bigg\Vert \< x \>^{- 3 / 2 - \varepsilon} \left( \begin{array}{cc} 0 & 0 \\ \widetilde{\partial} & 0 \end{array} \right) R^{2} ( z ) \< x \>^{- 3 / 2 - \varepsilon} \bigg\Vert_{\CL ( \CH^s , \CH^{s + 2} )} \lesssim 1 .
\end{equation*}
Interpolating with Lemma \ref{c2}, we get
\begin{equation} \label{c22}
\bigg\Vert \< x \>^{- 1 / 2 - \alpha - \varepsilon} \left( \begin{array}{cc} 0 & 0 \\ \widetilde{\partial} & 0 \end{array} \right) \big( R ( z ) - R ( z^{\prime} ) \big) \< x \>^{- 1 / 2 - \alpha - \varepsilon} \bigg\Vert_{\CL ( \CH^s , \CH^{s + 1} )} \lesssim \vert z - z^{\prime} \vert^{\alpha} .
\end{equation}
The same way,
\begin{equation} \label{c23}
\bigg\Vert \< x \>^{- 1 / 2 - \alpha - \varepsilon} \big( R ( z ) - R ( z^{\prime} ) \big) \left( \begin{array}{cc} 0 & 0 \\ \widetilde{\partial}^* & 0 \end{array} \right) \< x \>^{- 1 / 2 - \alpha - \varepsilon} \bigg\Vert_{\CL ( \CH^s , \CH^{s + 1} )} \lesssim \vert z - z^{\prime} \vert^{\alpha} .
\end{equation}

By Lemma \ref{c1}, we can write
\begin{equation}
R^k ( z ) - R^k ( z^{\prime} ) = \sum_{\text{\rm{finite}}} \sum_{j = 0}^{n} M_0 ( z )V \cdots V \big( M_j (z) - M_j ( z^{\prime} ) \big) V \cdots M_n ( z^{\prime} ) .
\end{equation}
Now, the rest of the proof is similar to the one of Theorem \ref{b16} and we omit the details. The difference is that we add an additional $\< x \>^{- \alpha}$ on the left and on the right of $( M_j (z) - M_j ( z^{\prime} ) )$ and that we use \eqref{c20}, \eqref{c21}, \eqref{c22} and \eqref{c23} instead of Proposition \ref{b1}, Proposition \ref{b13}, Lemma \ref{c2} and \eqref{c13} to estimate this term.
\end{proof}

\section{Proof of the main theorem} \label{s5}

In this part, we deduce Theorem \ref{b17} from the smoothness of the weighted resolvent obtained in Section \ref{s4}. First note that for $\mu < 2$, this theorem follows from \cite{BoHa10_03}. Indeed, under the assumption $\rho > 0$, it is proved in \cite[Theorem 1 $i)$]{BoHa10_03} that
\begin{equation*}
\Big\Vert \< x \>^{1 - d} e^{- i t G} \chi ( G ) \< x \>^{1 - d} \Big\Vert \lesssim \< t \>^{1 - d + \varepsilon} .
\end{equation*}
On the other hand, \cite[Lemma 4.2]{BoHa10_01} gives $\Vert \< x \>^{- 1 / 2 - \varepsilon} u \Vert \lesssim \Vert P^{1 / 4} u \Vert$, and then
\begin{equation*}
\Big\Vert \< x \>^{- 1 / 2 - \varepsilon} e^{- i t G} \chi ( G ) \< x \>^{- 1 / 2 - \varepsilon} \Big\Vert \lesssim 1 .
\end{equation*}
Interpolating the two previous estimates yields Theorem \ref{b17} for $\mu < 2$.

In the sequel, we assume that $\mu \geq 2$. Thus, we can apply Corollary \ref{c26} with $k = \lfloor \mu \rfloor + 1 \geq 3$. Using Stone's formula and integrating by parts, we get
\begin{align}
\< x \>^{- \mu - 1 - \varepsilon} & e^{- i t G} \chi ( G ) \< x \>^{- \mu - 1 - \varepsilon}     \nonumber \\
={}& \frac{1}{2 \pi i} \int \chi ( \lambda ) e^{- i t \lambda} \< x \>^{- \mu - 1 - \varepsilon} \big( R ( \lambda + i 0 ) - R ( \lambda - i 0 ) \big) \< x \>^{- \mu - 1 - \varepsilon} d \lambda   \nonumber \\
={}& \frac{1}{2 \pi i} \frac{1}{( i t )^{\lfloor \mu \rfloor}} \sum_{\pm} \sum_{j = 1}^{\lfloor \mu \rfloor + 1} \pm C_{\lfloor \mu \rfloor}^{j - 1}
\int \chi_{j} ( \lambda ) e^{- i t \lambda} \< x \>^{- \mu - 1 - \varepsilon} R^{j} ( \lambda \pm i 0 ) \< x \>^{- \mu - 1 - \varepsilon} d \lambda , \label{c24}
\end{align}
with $\chi_{j} = \partial^{\lfloor \mu \rfloor + 1 - j} \chi \in C^{\infty}_{0} ( \R )$. Moreover, mimicking the proof of \cite[Theorem 25]{FrGrSi08_01}, we obtain, for all $1 \leq j \leq \lfloor \mu \rfloor + 1$,
\begin{align}
A : ={}& \int \chi_{j} ( \lambda ) e^{- i t \lambda} \< x \>^{- \mu - 1 - \varepsilon} R^{j} ( \lambda \pm i 0 ) \< x \>^{- \mu - 1 - \varepsilon} d \lambda  \nonumber \\
={}& \int \chi_{j} ( \lambda + \pi / t ) e^{- i t ( \lambda + \pi / t )} \< x \>^{- \mu - 1 - \varepsilon} R^{j} ( \lambda + \pi / t \pm i 0 ) \< x \>^{- \mu - 1 - \varepsilon} d \lambda   \nonumber \\
={}& - \int \chi_{j} ( \lambda + \pi / t ) e^{- i t \lambda} \< x \>^{- \mu - 1 - \varepsilon} R^{j} ( \lambda + \pi / t \pm i 0 ) \< x \>^{- \mu - 1 - \varepsilon} d \lambda    \nonumber  \\
={}& - A + \int \big( \chi_{j} ( \lambda ) - \chi_{j} ( \lambda + \pi / t ) \big) e^{- i t \lambda} \< x \>^{- \mu - 1 - \varepsilon} R^{j} ( \lambda \pm i 0 ) \< x \>^{- \mu - 1 - \varepsilon} d \lambda   \nonumber \\ 
&+ \int \chi_{j} ( \lambda + \pi / t ) e^{- i t \lambda} \< x \>^{- \mu - 1 - \varepsilon} \big( R^{j} ( \lambda \pm i 0 ) - R^{j} ( \lambda + \pi / t \pm i 0 ) \big) \< x \>^{- \mu - 1 - \varepsilon} d \lambda   \nonumber  \\
={}& \CO \big( t^{\lfloor \mu \rfloor - \mu} \big) ,    \label{c25}
\end{align}
since $\lambda \mapsto \< x \>^{- \mu - 1 - \varepsilon} R^{j} ( \lambda \pm i 0 ) \< x \>^{- \mu - 1 - \varepsilon}$ (and of course $\lambda \mapsto \chi_{j} ( \lambda )$) is H\"{o}lder continuous of order $\mu - \lfloor \mu \rfloor$ thanks to Corollary \ref{c26}. Then, \eqref{c24} and \eqref{c25} imply part $i)$ of Theorem \ref{b17} in the case $d \geq 3$ odd. This argument gives also the improvement in dimension $d=3$. In order to prove part $ii)$ of the theorem, it is sufficient to use the high energy estimates of \cite[Theorem 5 $ii)$]{BoHa10_03} as well as the formula \eqref{b5}.

\bibliographystyle{amsplain}

\begin{thebibliography}{10}

\bibitem{AgCo71_01}
J.~Aguilar and J.~M. Combes, \emph{A class of analytic perturbations for
  one-body {S}chr\"odinger {H}amiltonians}, Comm. Math. Phys. \textbf{22}
  (1971), 269--279.

\bibitem{Al10_01}
S.~Alinhac, \emph{Geometric analysis of hyperbolic differential equations: an
  introduction}, London Mathematical Society Lecture Note Series, vol. 374,
  Cambridge University Press, 2010.

\bibitem{BoHa10_03}
J.-F. Bony and D.~H{\"a}fner, \emph{Local energy decay for several evolution
  equations on asymptotically {E}uclidean manifolds}, preprint arXiv:1008.2357
  (2010).

\bibitem{BoHa10_02}
J.-F. Bony and D.~H{\"a}fner, \emph{Low frequency resolvent estimates for long range perturbations
  of the {E}uclidean {L}aplacian}, Math. Res. Lett. \textbf{17} (2010), no.~2,
  301--306.

\bibitem{BoHa10_01}
J.-F. Bony and D.~H{\"a}fner, \emph{The semilinear wave equation on asymptotically {E}uclidean
  manifolds}, Comm. Partial Differential Equations \textbf{35} (2010), 23--67.

\bibitem{Bo10_01}
J.-M. Bouclet, \emph{Low frequency estimates and local energy decay for
  asymptotically euclidean {L}aplacians}, preprint arXiv:1003.6016 (2010).

\bibitem{Bu98_01}
N.~Burq, \emph{D\'ecroissance de l'\'energie locale de l'\'equation des ondes
  pour le probl\`eme ext\'erieur et absence de r\'esonance au voisinage du
  r\'eel}, Acta Math. \textbf{180} (1998), no.~1, 1--29.

\bibitem{DoMcTh66_01}
C.~L. Dolph, J.~B. McLeod, and D.~Thoe, \emph{The analytic continuation of the
  resolvent kernel and scattering operator associated with the {S}chroedinger
  operator}, J. Math. Anal. Appl. \textbf{16} (1966), 311--332.

\bibitem{FrGrSi08_01}
J.~Fr{\"o}hlich, M.~Griesemer, and I.~M. Sigal, \emph{Spectral theory for the
  standard model of non-relativistic {QED}}, Comm. Math. Phys. \textbf{283}
  (2008), no.~3, 613--646.

\bibitem{GeSi92_01}
C.~G{\'e}rard and I.~M. Sigal, \emph{Space-time picture of semiclassical
  resonances}, Comm. Math. Phys. \textbf{145} (1992), no.~2, 281--328.

\bibitem{Ho83_01}
L.~H{\"o}rmander, \emph{The analysis of linear partial differential operators.
  {I}}, Grundlehren der Mathematischen Wissenschaften, vol. 256,
  Springer-Verlag, 1983.

\bibitem{Ho97_01}
L.~H{\"o}rmander, \emph{Lectures on nonlinear hyperbolic differential equations},
  Math\'ematiques \& Applications, vol.~26, Springer-Verlag, 1997.

\bibitem{Hu86_01}
W.~Hunziker, \emph{Distortion analyticity and molecular resonance curves}, Ann.
  Inst. H. Poincar\'e Phys. Th\'eor. \textbf{45} (1986), no.~4, 339--358.

\bibitem{JeKa79_01}
A.~Jensen and T.~Kato, \emph{Spectral properties of {S}chr\"odinger operators
  and time-decay of the wave functions}, Duke Math. J. \textbf{46} (1979),
  no.~3, 583--611.

\bibitem{Ki11_01}
Y.~Kian, \emph{On the meromorphic continuation of the resolvent for the wave
  equation with time-periodic perturbations and applications}, preprint
  arXiv:1103.2530 (2011).

\bibitem{Kl01_01}
S.~Klainerman, \emph{A commuting vectorfields approach to {S}trichartz-type
  inequalities and applications to quasi-linear wave equations}, Internat.
  Math. Res. Notices (2001), no.~5, 221--274.

\bibitem{LaPh89_01}
P.~Lax and R.~Phillips, \emph{Scattering theory}, second ed., Pure and Applied
  Mathematics, vol.~26, Academic Press Inc., 1989, With appendices by C.
  Morawetz and G. Schmidt.

\bibitem{Ra69_01}
J.~Ralston, \emph{Solutions of the wave equation with localized energy}, Comm.
  Pure Appl. Math. \textbf{22} (1969), 807--823.

\bibitem{Ra78_01}
J.~Rauch, \emph{Local decay of scattering solutions to {S}chr\"odinger's
  equation}, Comm. Math. Phys. \textbf{61} (1978), no.~2, 149--168.

\bibitem{SaZw95_01}
A.~S{\'a}~Barreto and M.~Zworski, \emph{Existence of resonances in three
  dimensions}, Comm. Math. Phys. \textbf{173} (1995), no.~2, 401--415.

\bibitem{Sj07_01}
J.~Sj{\"o}strand, \emph{Lectures on resonances}, preprint on {\it
  http://www.math.polytechnique.fr/$\sim$sjoestrand} (2007), 1--169.

\bibitem{SjZw91_01}
J.~Sj{\"o}strand and M.~Zworski, \emph{Complex scaling and the distribution of
  scattering poles}, J. Amer. Math. Soc. \textbf{4} (1991), no.~4, 729--769.

\bibitem{TaZw00_01}
S.-H. Tang and M.~Zworski, \emph{Resonance expansions of scattered waves},
  Comm. Pure Appl. Math. \textbf{53} (2000), no.~10, 1305--1334.

\bibitem{Va89_01}
B.~Va{\u\i}nberg, \emph{Asymptotic methods in equations of mathematical
  physics}, Gordon \& Breach Science Publishers, 1989.
\end{thebibliography}
\providecommand{\bysame}{\leavevmode\hbox to3em{\hrulefill}\thinspace}
\providecommand{\MR}{\relax\ifhmode\unskip\space\fi MR }
\providecommand{\MRhref}[2]{%
  \href{http://www.ams.org/mathscinet-getitem?mr=#1}{#2}
}
\providecommand{\href}[2]{#2}

\end{document}